\documentclass[11pt]{amsart}
\usepackage[utf8]{inputenc}
\usepackage{latexsym,url,xspace}
\usepackage{wasysym} 
\usepackage[all]{xy}
\usepackage{hyperref}
\usepackage{amssymb}

\usepackage{float,caption}
\newfloat{diagram}{htb}{dgm}[section]

\numberwithin{equation}{section}

\newtheorem{theo}{Theorem}
\swapnumbers

\theoremstyle{plain}
\newtheorem{prp}[subsection]{Proposition}
\newtheorem{lem}[subsection]{Lemma}

\theoremstyle{definition}
\newtheorem{defn}[subsection]{Definition}

\newtheorem{rems}[subsection]{Remarks}

\newcommand{\bsm}{\begin{smallmatrix}}
\newcommand{\esm}{\end{smallmatrix}}
\newcommand{\smp}[1]{\left( \bsm #1 \esm \right)}

\newcommand{\ie}{\emph{i.e.}\@\xspace}
\newcommand{\df}{\colon}
\newcommand{\ra}{\rightarrow}
\newcommand{\xra}[1]{\xrightarrow{#1}}
\newcommand{\xla}[1]{\xleftarrow{#1}}
\newdir{ >}{{}*!/-5pt/\dir{>}} 
\newcommand{\id}{1\kern -.35em 1}

\newcommand{\dA}{\mathsf{A}}
\newcommand{\dD}{\mathsf{D}}
\newcommand{\dE}{\mathsf{E}}

\newcommand{\Xx}{\mathbb{X}}

\newcommand{\alp}{\alpha}
\newcommand{\bet}{\beta}

\newcommand{\del}{\delta}

\newcommand{\lam}{\lambda}
\newcommand{\iot}{\iota}
\newcommand{\vph}{\varphi}

\newcommand{\sig}{\sigma}

\newcommand{\Lam}{\Lambda}
\newcommand{\Hom}{\operatorname{Hom}}
\newcommand{\sHom}{\operatorname{\underline{Hom}}}
\newcommand{\Ext}{\operatorname{Ext}}
\newcommand{\End}{\operatorname{End}}

\newcommand{\Ker}{\operatorname{Ker}}

\newcommand{\op}{{\operatorname{op}}}

\newcommand{\add}{\operatorname{add}}

\newcommand{\ka}{k}                 

\newcommand{\Zet}{\mathbb{Z}}

\newcommand{\Sig}{\Sigma}          
\newcommand{\Ome}{\Omega}          
\newcommand{\Se}{\mathbb{S}}       

\newcommand{\rMd}[1]{\operatorname{Mod}#1}             
\newcommand{\rmd}[1]{\operatorname{mod}#1}             
\newcommand{\smd}[1]{\operatorname{\underline{mod}}#1} 
\newcommand{\rin}[1]{\operatorname{inj}#1}             
\newcommand{\rpr}[1]{\operatorname{proj}#1}
\newcommand{\Ab}{\operatorname{Ab}} 
\newcommand{\Fo}{\mathcal{F}}       

\newcommand{\cC}{\mathcal{C}}
\newcommand{\Cac}{\mathcal{C}_\text{ac}} 
\newcommand{\Hac}{\mathcal{H}_\text{ac}} 
\newcommand{\cE}{\mathcal{E}}       
\newcommand{\sE}{{\underline{\cE}}}
\newcommand{\cP}{\mathcal{P}}       
\newcommand{\Tr}{\mathcal{T}}

\newcommand{\cc}{\mathcal{C}}

\newcommand{\foa}{\pentagon}             

\newcommand{\bul}{\bullet}
\newcommand{\nSig}{n\text{-}\Sig}

\title{$n$-Angulated Categories}

\author{Christof Geiß}
\address{Instituto de Matemáticas, Universidad Nacional Autonoma de México,
Ciudad Universitaria,
04510 México D.F., Mexico}
\email{christof@math.unam.mx}
\urladdr{www.matem.unam.mx/christof/}

\author{Bernhard Keller}
\address{UFR de Mathématiques, Université Denis Diderot - Paris 7,
Institut de Math\'ematiques, UMR 7586 du CNRS,
B\^atiment Chevaleret, 175 rue du Chevaleret, 75013 Paris, FRANCE}
\email{keller@math.jussieu.fr}
\urladdr{www.math.jussieu.fr/~keller/}

\author{Steffen Oppermann}
\address{Institutt for matematiske fag, NTNU, 7491 Trondheim, Norway}
\email{steffen.oppermann@math.ntnu.no}
\urladdr{www.math.ntnu.no/~opperman/}

\date{May 25, 2011}
\begin{document}

\begin{abstract}
We define $n$-angulated categories by modifying the axioms of triangulated
categories in a natural way. We show that Heller's parametrization of
pre-triangulations extends to pre-$n$-angulations. We obtain a large class
of examples of $n$-angulated categories by considering $(n-2)$-cluster tilting
subcategories of triangulated categories which are stable under the $(n-2)$nd
power of the suspension functor. As an application, we show how
$n$-angulated Calabi-Yau categories yield triangulated
Calabi-Yau categories of higher Calabi-Yau dimension. Finally,
we sketch a link to algebraic geometry and string theory.
\end{abstract}

\maketitle
\section*{Introduction}

\subsection{Context}
Triangulated categories were invented at the end of the 1950s by
Grothen\-dieck-Verdier~\cite{Verdier77} and, independently, Puppe~\cite{Puppe62a}.
Their aim was to axiomatize
the properties of derived categories respectively of stable homotopy
categories. In the case of the derived category, the triangles are the
`shadows' of the $3$-term exact sequences of complexes. Longer exact
sequences of complexes are splicings of $3$-term exact sequences and
thus they also have their shadows in the derived category, for example
in the form of the higher octahedra described in Remark~1.1.14 of
\cite{BeilinsonBernsteinDeligne82}. Cluster tilting theory (as
developed in \cite{BuanMarshReinekeReitenTodorov06, GeissLeclercSchroeer05, IyamaYoshino08} and many other articles)
and in particular Iyama's higher Auslander-Reiten-theory (see \cite{Iyama08,Iyama07a})
have lead to the surprising discovery that
there is a large class of categories which are naturally inhabited by
shadows of $n$-term exact sequences without being home to shadows of
$3$-term exact sequences. In this paper, our main aims are to
\begin{itemize}
\item axiomatize this remarkable class of categories by
introducing the new notion of {\em $n$-angulated category} and
\item construct large classes of examples using cluster tilting theory.
\end{itemize}
The main thrust of this paper is thus foundational. However, as we
will see, this apparently dry subject matter is linked to some
exciting developments in algebraic geometry
and string theory \cite{KontsevichSoibelman08, KontsevichSoibelman10, CecottiNeitzkeVafa10,GaiottoMooreNeitzke10}.

\subsection{Contents}
In creating our axiomatic framework
we found that mostly, the axioms of triangulated categories generalize
well but that special care has to be taken with the notion of isomorphism
of $n$-angles and with the octahedral axiom, cf.\ Section~\ref{s:axioms}.

In Section~\ref{s:parametrization}, we pursue the foundational analogy with the
triangulated case by generalizing a result of Heller \cite{Heller68a}: We 
parametrize the set of equivalence classes of pre-$n$-angulations 
(a pre-$n$-angulation need not satisfy the generalized octahedral axiom) on a 
given additive category $\mathcal{F}$ with a given higher suspension functor 
$\Sigma$.
Namely, we show that if it is not empty, this set has a simply transitive
action by the automorphism group of the $n$th suspension
in the stable category associated with the Frobenius category of
finitely presented $\mathcal{F}$-modules.

In Section~\ref{s:standard-construction}, we show how to
construct $n$-angulated categories inside triangulated categories.
Namely, if $\mathcal{T}$ is a triangulated category with
suspension functor $\Sigma$ and
$\mathcal{S}$ is an $(n-2)$-cluster tilting subcategory
in the sense of Iyama \cite{Iyama07a} which is stable under
$\Sigma^{n-2}$, then $\mathcal{S}$ naturally becomes an
$n$-angulated category with higher suspension functor $\Sigma^{n-2}$.
This yields a large supply of interesting $n$-angulated categories
and is one of the main results of this paper.

In Section~\ref{s:calabi-yau-properties}, we show that if
$\mathcal{F}$ is an $n$-angulated category which is Calabi-Yau
of Calabi-Yau dimension $d$, then the stable category of
finitely presented $\mathcal{F}$-modules is (triangulated and)
Calabi-Yau of Calabi-Yau dimension $nd-1$. In particular,
we can construct $3$-Calabi-Yau categories from $4$-angulated
categories which are $1$-Calabi-Yau. This generalizes
the construction in example~8.3.3 of \cite{Kell05}.
It was one of our original motivations for studying
$n$-angulated categories.

We conclude by presenting several classes of examples
in Section~\ref{s:examples} and by sketching a link
to algebraic geometry and string theory in
\ref{ss:links-alg-geom-string-theory}.

\section*{Acknowledgments}
C.~G. and B.~K. thank Idun Reiten and Claus Michael Ringel
for organizing the 2005 Oberwolfach meeting on representation theory
of finite-dimensional algebras, where the results of this paper were
first presented in rudimentary form.

C.~G. acknowledges partial support from PAPIIT grant IN103507-2
and CONACYT grant 81948.

\section{Axioms}
\label{s:axioms}
\begin{defn} Let $\Fo$ be an additive category with an automorphism
$\Sig$, and $n$ an integer greater or equal than three. A sequence of morphisms
of $\Fo$
\[
X_\bul:=
(X_1\xra{\alp_1} X_2\xra{\alp_2}\cdots\xra{\alp_{n-1}} X_n\xra{\alp_n}\Sig X_1)
\]
is an $\nSig$-\emph{sequence}. Its \emph{left rotation} is the
$\nSig$-sequence
\[
(X_2 \xra{\alp_2} X_3 \xra{\alp_3}\cdots\xra{\alp_n}
\Sig X_1\xra{(-1)^n\Sig\alp_1} \Sig X_2).
\]
The $\nSig$-sequences of the form
$(TX)_\bul:=(X\xra{\id_X} X \xra{} 0 \xra{}\cdots  \xra{}0 \xra{} \Sig X)$
for $X\in\Fo$, and their rotations, are called \emph{trivial}.

An $\nSig$-sequence $X_\bul$ is \emph{exact} if the induced sequence
\[
\Fo(-,X_\bul)\df\quad\cdots\ra\Fo(-,X_1)\ra\Fo(-,X_2)\ra\cdots\ra\Fo(-,X_n)\ra
\Fo(-,\Sig X_1)\ra\cdots
\]
of representable functors $\Fo^\op\ra\Ab$ is exact. In particular, the trivial
$\nSig$-sequences are exact.

A \emph{morphism} of $\nSig$-sequences is given by
a sequence of morphisms $\vph=(\vph_1,\vph_2,\ldots,\vph_n)$ such
that the following diagram commutes:
\[
\def\objectstyle{\scriptstyle}
\xymatrix{
X_\bul\ar[d]^{\vph_\bul}& X_1\ar[r]^{\alp_1}\ar[d]^{\vph_1} & X_2\ar[r]^{\alp_2}\ar[d]^{\vph_2} &
X_3\ar@{.}[r]\ar[d]^{\vph_3} & X_n\ar[r]^{\alp_n}\ar[d]^{\vph_n} &\Sig X_1\ar[d]^{\Sig\vph_1}  \\
Y_\bul&Y_1\ar[r]^{\bet_1} &Y_2\ar[r]^{\bet_2} & Y_3\ar@{.}[r] &Y_n\ar[r]^{\bet_n} &\Sig Y_n
}
\]
In this situation we call $\vph$ a {\em weak isomorphism} if for some
$1\leq i\leq n$ both $\vph_i$ and $\vph_{i+1}$ (with $\vph_{n+1}:=\Sig\vph_1$)
are isomorphisms. Slightly abusing terminology we will say that
two $\nSig$-sequences are {\em weakly isomorphic} if they are
linked by a finite zigzag of weak isomorphisms.

We call a collection $\foa$ of $\nSig$-sequences a
{\em (pre-) $n$-angulation of  $(\Fo,\Sig)$}
and its elements $n$-\emph{angles}
if $\foa$ fulfills the following axioms:
\begin{itemize}
\item[(F1)]
  \begin{itemize}
  \item[(a)]
    $\foa$ is closed under direct sums and under taking direct summands.
  \item[(b)]
For all $X\in\Fo$, the trivial $\nSig$ sequence $(TX)_\bul$ belongs to $\foa$.
  \item[(c)]
    For each morphism $\alp_1\df X_1\ra X_2$ in $\Fo$, there exists an
    $n$-angle whose first morphism is $\alp_1$.
  \end{itemize}
\item[(F2)]
An $\nSig$-sequence $X_\bul$ belongs to $\foa$
if and only if its left rotation
\[
(X_2 \xra{\alp_2} X_3 \xra{\alp_3}\cdots\xra{\alp_n}
\Sig X_1\xra{(-1)^n\Sig\alp_1} \Sig X_2)
\]
belongs to $\foa$.
\item[(F3)] Each commutative diagram
\[
\def\objectstyle{\scriptstyle}
\xymatrix{
X_1\ar[r]^{\alp_1}\ar[d]^{\vph_1} & X_2\ar[r]^{\alp_2}\ar[d]^{\vph_2} &
X_3\ar@{.}[r]\ar@{.>}[d]^{\vph_3} & X_n\ar[r]^{\alp_n}\ar@{.>}[d]^{\vph_n}
&\Sig X_1\ar[d]^{\Sig\vph_1}  \\
Y_1\ar[r]^{\bet_1} &Y_2\ar[r]^{\bet_2} & Y_3\ar@{.}[r] &Y_n\ar[r]^{\bet_n} &\Sig Y_1
}
\]
with rows in $\foa$ can be completed to a morphism of $\nSig$-sequences.

\end{itemize}
If $\foa$ moreover fulfills the following axiom, it is called an
$n$-\emph{angulation} of $(\Fo,\Sig)$:
\begin{itemize}
\item[(F4)]
In the situation of (F3) the morphisms $\vph_3,\vph_4,\ldots,\vph_n$ can be
chosen such that the cone $C(\vph_\bul)\df$
\[
\xymatrix{
X_2\oplus Y_1\quad
\ar[r]^{\left(\bsm -\alp_2& 0\\\phantom{-}\vph_2&\bet_1\esm\right)}&
\quad X_3\oplus Y_2\quad
\ar[r]^{\left(\bsm -\alp_3& 0\\\phantom{-}\vph_3&\bet_2\esm\right)}&
\quad \cdots\quad
\ar[r]^{\left(\bsm -\alp_n& 0\\\phantom{-}\vph_n&\bet_{n-1}\esm\right)}&
\qquad\Sig X_1\oplus Y_n\quad
\ar[r]^{\left(\bsm -\Sig\alp_1& 0\\\phantom{-}\Sig\vph_1& \bet_n\esm\right)}&
\quad \Sig X_2\oplus \Sig Y_1
}
\]
belongs to $\foa$.
\end{itemize}
\end{defn}

\begin{rems}
(a) For (pre-) triangulated categories it is possible to demand in
axiom (TR2), the ``model'' for our (F2) only one direction, see
for example~\cite[I.1.3]{happ88}.
We need here however both directions. One reason for this is that
for $n\geq 4$ our Axiom~(F3)
is weaker than its ``model'' for triangulated categories. For example, if
$\vph_1$ and $\vph_2$ are isomorphisms it does {\em not} follow that
$\vph_3,\ldots,\vph_n$ are isomorphisms.

(b) Our axiom (F4) is inspired by Neeman's version of
``octahedral axiom (TR4)'', see for example~\cite[1.3, 1.8]{neeman01}.

(c) If $(\Fo,\Sig,\foa)$ is (pre-) $n$-angulated, the opposite category
$(\Fo^\op,\Sig^{-1})$ is (pre-) $n$-angulated with $n$-angles
\[
\Sig^{-1}X_n\xla{(-1)^n\Sig^{-1}\alp_n} X_1\xla{\alp_1} X_2\xla{}\cdots
\xla{\alp_{n-1}}X_n
\]
corresponding to the $n$-angles in $\foa$.

(d) In examples, $n$-angulated categories come frequently with a
self-equivalence $\Sig$. However, in analogy with~\cite[Sec.~2]{kevo88}
we may assume without loss of generality that $\Sig$ is an automorphism.
This has the advantage of a less heavy notation.
\end{rems}

\subsection{Periodic Complexes} \label{ssec:percpx}
Consider a complex $M_\bul=(M_k,\del_k)_{k\in\Zet}$ in $\rMd{\Fo}$,
the abelian category of functors $\Fo^\op\ra\Ab$.
For later use, we fix factorizations
\[
\xymatrix{
M_k\ar[rr]^{\del_k}\ar@{->>}[rd]_{\del'_k} && M_{k+1}\\
& K_k\ar@{>->}[ru]_{\del''_k}
}
\]
with $\del'_k$ an epimorphism and $\del''_k$ a monomorphism
for all $k\in\Zet$. Recall that $M_k$ is contractible if there exists a
contraction, \ie a family of morphisms $\eta_k\df M_{k+1} \ra M_k$ such that
$\del_{k-1}\eta_{k-1}+\eta_k\del_k=\id_{M_k}$ for all $k$. Equivalently,
$M_\bul$ is exact, and there exist sections $\eta'_k$ for $\del'_k$  for all $k$.

The automorphism $\Sig$ of $\Fo$ induces an automorphism of $\rMd{\Fo}$
defined by $M\mapsto M\circ\Sig^{-1}$ which we denote also by $\Sig$.
We say that $M_\bul$ is $\nSig$-\emph{periodic} if
$M_{k+n+1}=\Sig M_k$ and $\del_{k+n+1}=\Sig\del_k$ for all $k\in\Zet$. A
$\nSig$-periodic complex $M_\bul$ is \emph{trivial} if it is
exact and there exists $l\in\{1,\ldots,n\}$ such that $M_k=0$ unless
$k\equiv l\pmod{n}$ or $k\equiv l+1\pmod{n}$.

Note that if an $\nSig$-periodic complex $M_\bul$ is contractible,
then it admits an $\nSig$-periodic contraction. In particular, it
is a finite direct sum of $\nSig$-periodic trivial complexes.
Indeed, in the above factorization, we may assume that
$\del'_{k+n+1}=\Sig\del'_k$ and $\del''_{k+n+1}=\Sig\del''_k$ for all $k$.
Since $M_\bul$ is contractible, we find sections $\eta'_k$ to $\del'_k$
for $k=1,2,\ldots, n$. Then, $\Sig^l\eta'_k$ is a section to $\del'_{l(n+1)+k}$
for all $l\in\Zet$ and $1\leq k\leq n$. This yields an $\nSig$-periodic
contraction.

\begin{lem} \label{lem:wiso}
Let $(\Fo,\Sig,\foa)$ be a pre-$n$-angulated category. If $X_\bul$ and $Y_\bul$
are two weakly isomorphic exact $\nSig$-sequences, then $X_\bul$ belongs to
$\foa$ if and only if $Y_\bul$ belongs to $\foa$.
\end{lem}

\begin{proof}
(1) Let $X_\bul$ be an exact $\nSig$-sequence. Then it is easy to see
that the cone over the identity $C(\id_X)$ is a finite direct sum of trivial
complexes. So $C(\id_X)$ belongs to $\foa$.

(2) Let $\vph_\bul\df X_\bul\ra Y_\bul$ be a weak isomorphism between exact
$\nSig$-sequences. Then the cone $C(\vph_\bul)$ admits a periodic
contraction.
Equivalently, the $\nSig$-periodic complex $\Fo(-,C(\vph_\bul))$
is a direct sum of trivial $\nSig$-periodic complexes in
$\mod{\Fo}$.

Indeed, the induced morphism of complexes
$\Fo(-,\vph_\bul)\df\Fo(-,X_\bul)\ra\Fo(-,Y_\bul)$ is a homotopy
equivalence since for appropriate truncations those complexes can be
seen as the projective resolution of isomorphic modules with the isomorphism
induced by $\vph_\bul$.  Our claim now follows from~\ref{ssec:percpx} since
the cone over a homotopy equivalence is contractible.

(3) For a weak isomorphism $\vph_\bul\df X_\bul\ra Y_\bul$, consider the
following sequence of $\nSig$-periodic, exact complexes:
\[
0\ra \Fo(-,X_\bul)\ra \Fo(-,Y)\oplus \Fo(-,C(\id_X))\xra{p} 
\Fo(-,C(\vph_\bul))\ra 0,
\]
it is exact for the usual degree-wise split exact structure on the category
of complexes. From~(2), it follows that $p$ admits an $\nSig$-periodic
section. Now, it follows from (F1) and (1) that in case $Y_\bul\in\foa$
also $X_\bul\in\foa$. A similar construction shows that in case
 $X_\bul\in\foa$ implies $Y_\bul\in\foa$.
\end{proof}

\begin{prp} \label{axi:prp}
Let $\Fo=(\Fo,\Sig,\foa)$ be a pre-$n$-angulated category.
Then the following hold:
\begin{itemize}
\item[(a)]
All $n$-angles are exact.
\item[(b)]
$\rmd{\Fo}$, the category of finitely presented functors
$\Fo^{\text{op}}\ra\Ab$ is an abelian Frobenius category. Moreover, if
$\Fo$ has split idempotents, the map $X\mapsto \Fo(-,X)$ induces an
equivalence from $\Fo$ to $\rin{\Fo}$, the full subcategory of $\rmd{\Fo}$ which
consists of the injective objects.
\item[(c)]
If $\foa'\subset\foa$ is an other pre-$n$-angulation of $(\Fo,\Sig)$, then
$\foa'=\foa$.
\end{itemize}
\end{prp}

\begin{proof}
(a) This is almost verbatim the same argument as the corresponding statement
for pre-triangulated categories,
see for example~\cite[1.1.3 and 1.1.10]{neeman01}.

(b) $\rmd{\Fo}$ has clearly cokernels, and it has kernels since by (a)
$\Fo$ has weak kernels~\cite{Cartan53}.
In order to conclude 
that each projective is also injective we show 
$\Ext^i_\Fo(M, \Fo(-,T))=0$ for each $M\in\rmd{\Fo}$ and $T\in\Fo$. 
Indeed, for a projective  presentation
\[
\Fo(-,X_{n-1})\xra{\Fo(-,\alp_{n-1})}\Fo(-,X_n)\ra M\ra 0
\]
$\alp_{n-1}$ can be embedded in an $n$-angle $X_\bul$ by (F1c) and (F2).
By part~(a) of the proposition we obtain a (periodic)
projective resolution $P^\bullet$ of $M$. 
By the Yoneda lemma and part~(a) of the proposition applied to
$\Fo^\op$ we conclude that $H^i(P^\bullet,\Fo(-,T))=0$ for all $i\geq 1$.

(c) Let $X_\bul$ be in $\foa$.
By~(F1c) there exists a 4-angle
$X_1\xra{\alp}X_2\xra{\alp'_2}X'_2\xra{\alp'_3}\cdots\xra{\alp'_n}\Sig X_1$ in
$\foa'$.
Since $\foa'\subset\foa$, the pair $(\id_{X_1},\id_{X_2})$ can be completed
by~(F3) to a morphism of $n$-angles. But this is a weak isomorphism, 
so the first $n$-angle belongs already to $\foa'$ by Lemma~\ref{lem:wiso}.
\end{proof}

\section{The class of $n$-angulations}
\label{s:parametrization}
In this section, we assume that $\Fo$ has split idempotents.
Closely following~\cite[\S 16]{Heller68a} we study the possible
pre-$n$-angulations for $(\Fo,\Sig)$.

\subsection{Injective Resolutions and Triangle Functors}
Let $(\Fo,\Sig,\foa)$ be a pre-$n$-angulated category.
Thus, by Proposition~\ref{axi:prp} $\rmd{\Fo}$ is a Frobenius category
and we may identify $\Fo$ with the subcategory $\rin{\Fo}$ of injectives
in $\rmd{\Fo}$. For each $M\in\rmd{\Fo}$, we choose a short exact sequence
\begin{equation} \label{eq:std1}
0\ra M \xra{\iot_M} I_M\xra{\pi_M} \Ome^{-1} M\ra 0 \;\; ,
\end{equation}
where $I_M$ belongs to $\Fo$ identified with $\rin(\Fo)$.
By splicing together the short exact sequences
from~\eqref{eq:std1}
we obtain a standard injective resolution
\begin{equation} \label{eq:std2}
\hat{I}_M\df\quad
I_M\xra{\nu_{M,1}} I_{\Ome^{-1}}\xra{\nu_{M,2}} I_{\Ome^{-2}M}\ra\cdots
\end{equation}
of $M$. Since the automorphism $\Sig$ of $\rmd{\Fo}$ is
an exact functor, we may assume $I_{\Sig M}=\Sig I_M$, and we obtain an
isomorphism
\[
\sig_M\df \Sig\Ome^{-1} M\ra \Ome^{-1}\Sig M.
\]
Next, the above short exact sequence induces a self equivalence $\Ome^{-1}$
of the stable category $\smd{\Fo}$. Recall that $\smd{\Fo}$ is a triangulated
category with suspension $\Ome^{-1}$. Moreover, $(\Sig,\sig)$ and
$(\Ome^{-n},(-1)^n\id_{\Ome^{-n-1}})$ are triangle functors on $\smd{\Fo}$.

Finally, if
\[
X_\bul: X_1\xra{\alp_1} X_2\xra{\alp_2}\cdots
\xra{\alp_{n-1}}X_n\xra{\alp_n}\Sig X_1
\]
is an exact $\nSig$-sequence in $\Fo$, we may interpret it as the beginning
of an injective resolution of $\Ker(\alp_1)\in\rmd{\Fo}$. Thus, we can lift
$\id_{\Ker{\alp_1}}$ to a unique, up to homotopy, morphism of complexes
from $X_\bul$ to $\hat{I}_{\Ker{\alp_1}}$. From the factorization
\[\xymatrix{
X_n\ar[rr]^{\alp_n}\ar@{->>}[rd]&&\Sig X_1 \;\;, \\
&\Sig\Ker{\alp_1}\ar@{>->}[ru]}
\]
we obtain a morphism
\[
\del_{X_\bul}\df \Sig\Ker(\alp_1)\ra \Ome^{-n}\Ker(\alp_1),
\]
which becomes an isomorphism in $\smd{\Fo}$.

\begin{lem} \label{lem:class1}
Associated to $\foa$ we have a natural isomorphism
\[
\del(\foa)\df (\Sig,\sig)\ra (\Ome^{-n},(-1)^n\id_{\Ome^{-n-1}})
\]
of triangle functors.
\end{lem}

\begin{proof}
By Proposition~\ref{axi:prp} we find for each $M\in\rmd{\Fo}$ an
$n$-angle $X_\bul\in\foa$ such that $\Ker(\alp_1)=M$. We set
then $\del(\foa)_M=\del_{X_\bul}\in\sHom_\Fo(\Sig M,\Ome^{-n}M)$.
As in~\cite[p.53]{Heller68a} one verifies that this isomorphism does
not depend of the particular choice of $X_\bul\in\foa$.

It remains to show that we have for each $M$ a commutative diagram
\[\xymatrix{
\Sig\Ome^{-1}\ar[d]_{\del(\foa)_{\Ome^{-1} M}}\ar[rr]^{\sig_M}& &\Ome^{-1}\Sig M\ar[d]^{\Ome^{-1}\del(\foa)_M}\\
\Ome^{-n-1}M\ar[rr]_{(-1)^n\id_{\Ome^{-1-n}M}}&&\Ome^{-1-n} M
}
\]
in $\smd{\Fo}$. This is mutatits mutandis the same argument as in the proof
of~\cite[Lemma 16.3]{Heller68a}.
\end{proof}

\begin{lem} \label{lem:class2}
Suppose $\rmd{\Fo}$ is a Frobenius category.
Let $\Theta\df(\Sig,\sig)\ra (\Ome^{-n},(-1)^n\id_{\Ome^{-n-1}})$ be an
isomorphism of triangle functors in $\smd{\Fo}$.
Then the collection $\foa_\Theta$ of exact $\nSig$-sequences
$X_\bul=(X_1\xra{\alp_1}\cdots\xra{\alp_{n-1}} X_n\xra{\alp_n} \Sig X_1)$ in
$\Fo$ such that $\del_{X_\bul}=\Theta_{\Ker(\alp_1)}$ is a pre-$n$-angulation of
$(\Fo,\Sig)$.
\end{lem}

\begin{proof}
We verify the axioms (F1)-(F3) for $\foa_\Theta$.

(F1a) and (F1b) are trivially fulfilled.

(F1c): Let $\alp_1\in\Fo(X_1,X_2)$, and $A:=\Ker(\alp_1)\rmd{\Fo}$. Clearly,
we can find in $\rmd{\Fo}$ a commutative diagram with  exact rows
\[\xymatrix{
0\ar[r] &A\ar@{=}[d]\ar[r]^{\iot} &X_1\ar[d]\ar[r]^{\alp_1}&X_2\ar[d]\ar[r]^{\alp_2}& X_3\ar[d]\ar[r]^{\alp_3}&
\cdots \ar[r]^{\alp_{n-2}} &X_{n-1}\ar[d]\ar[r]^{\pi_n} &C\ar[d]^{\vph}\ar[r] &0\\
0\ar[r] &A\ar[r] & I_A\ar[r] & I_{\Ome^{-1}A}\ar[r]& I_{\Ome^{-2}A}\ar[r]&\cdots\ar[r]&I_{\Ome^{2-n}A}\ar[r]&\Ome^{1-n}A\ar[r] & 0
}\]
and $X_3,\ldots, X_{n-1}\in\Fo$. Note that the class of $\vph$ in $\smd{\Fo}$
is an isomorphism.

Next we find a commutative diagram with exact rows
\[\xymatrix{
0\ar[r] &C\ar@{=}[d]\ar[r]^{\iot_{n-1}} &X_n\ar[d]\ar[r]^{\pi_n} &\Sig A\ar[d]^{\theta}\ar[r]& 0\\
0\ar[r] &C\ar[d]_{\vph}\ar[r]& I_C\ar[d]\ar[r] &\Ome^{-1} C\ar[d]^{\Ome^{-1}\vph}\ar[r] & 0\\
0\ar[r] &\Ome^{1-n} A\ar[r] &I_{\Ome^{1-n}A}\ar[r] &\Ome^{-n}A\ar[r] & 0
}\]
such that $X_n\in\Fo$ and
$\Ome^{-1}\vph\theta=\Theta_A\in\sHom_\Fo(\Sig A,\Ome^{-n}A)$.
Then, $X_\bul$ with $\alp_{n-1}=\iot_{n-1}\pi_{n-1}$ and $\alp_n=(\Sig\iota)\pi_n$
is an $\nSig$-sequence with $\del_{X_\bul}=\Theta_A$.

(F2): This follows since $\Theta$ is an isomorphism of \emph{triangle} functors,
see the second part of the proof of Lemma~\ref{lem:class1}.

(F3): In the situation of (F3), let $A=\Ker(\alp_1)$ and $B=\Ker(\bet_1)$.
We clearly obtain a commutative diagram:
\[\xymatrix{
&X_1\ar[dd]_{\vph_1}\ar[r]^{\alp_1}&X_2\ar[dd]_{\vph_2}\ar[r]^{\alp_2}&X_3\ar[dd]_{\vph_3}\ar@{.}[r]&X_n\ar[dd]_{\vph_n}\ar[rr]^{\alp_n}\ar@{->>}[rd]^{\pi_n}&&\Sig X_n\ar[dd]^{\vph_{n+1}}\\
A\ar[dd]_{\vph}\ar[ru]^{\iot}& &                 &             &&\Sig A\ar[dd]_(.3){\psi}\ar@{>->}[ru]^{\Sig\iot}\\
&Y_1\ar[r]_{\bet_1}&Y_2\ar[r]_{\bet_2}&Y_3\ar@{.}[r]&Y_n\ar'[r][rr]_(.3){\bet_n}\ar@{->>}[rd]_{\rho_n}&&\Sig Y_n\\
B\ar[ru]_{\kappa}& &                 &             &&\Sig B\ar@{>->}[ru]_{\Sig\kappa}\\
}\]
By construction, $\del_{Y_\bul}\psi=(\Ome^{-n}\vph)\del_{X_\bul}$. On the other
hand, by the naturality of $\Theta$, we have
\[
\del_{Y_\bul}(\Sig\vph)=(\Ome^{-n}\vph)\del_{X_\bul}
\in\sHom_\Fo(\Sig A,\Ome^{-n} B).
\]
Thus, we have $\eta'\in\Hom_\Fo(\Sig A, Y_n)$ and
$\eta''\in\Hom_\Fo(\Sig X_n,\Sig B)$  such that
$\psi-\Sig\vph=\rho_n\eta'=\eta''(\Sig\iot)$. As a consequence, we
may replace in the above diagram $\psi$ by $\Sig\vph$,
$\vph_n$ by $\vph'_n=\vph_n-\eta'\pi_n$ and $\vph_{n+1}$ by
$\vph'_{n+1}=\vph_{n+1}-(\Sig\kappa)\eta''$.
Now, $(\vph'_{n+1}-\Sig\vph_1)(\Sig\iot)=\Sig(\vph\kappa -\vph_1\iot)=0$,
so that $(\vph_1,\vph_2,\ldots,\vph_{n-1},\vph'_n)$ is a morphism
of $\nSig$-sequences.
\end{proof}

\begin{prp}
Let $(\Fo,\Sig,\foa)$ be a pre-n-angulated category. Then the map $\del$
from the class of pre-n-angulations of $(\Fo,\Sig)$ to the class of
isomorphisms of triangle functors between $(\Sig,\sig)$ and
$(\Ome^{-n},(-1)^n\id_{\Ome^{-n-1}})$, defined in Lemma~\ref{lem:class1},
is a bijection. As a consequence, the group of automorphisms of the
triangle functor $(\Ome^{-n},(-1)^n\id_{\Ome^{-n-1}})$ acts simply transitively
on the class of pre-$n$-angulations of $(\Fo,\Sig)$.
\end{prp}

\begin{proof}
If $\Theta\df(\Sig,\sig)\ra (\Ome^{-n},(-1)^n\id_{\Ome^{-n-1}})$  is an
isomorphism of triangle functors, we have by Lemma~\ref{lem:class2}
$\del(\foa_\Theta)=\Theta$. If $\del(\foa')=\Theta$, we have by construction
$\foa'\subset\foa_\Theta$, but then $\foa'=\foa_\Theta$ by
Proposition~\ref{axi:prp}~(c).
\end{proof}

\section{Standard construction}
\label{s:standard-construction}
\begin{defn} \label{def.ct}
Let $\mathcal{T}$ be a triangulated category with suspension $\Sigma_3$.
A full subcategory $\mathcal{S} \subseteq \mathcal{T}$ is called
\emph{$d$-cluster tilting} if it is functorially 
finite~\cite[p.82]{AusSm80}, and
\begin{align*}
\mathcal{S} &=\{X\in\mathcal{T}\mid\mathcal{T}(\mathcal{S}, \Sigma_3^i X)= 0\;
\forall i \in \{1, \ldots, d-1\} \} \\
& =\{X \in\mathcal{T}\mid\mathcal{T}(X, \Sigma_3^i \mathcal{S})= 0 \;
\forall i  \in \{1, \ldots, d-1\} \}.
\end{align*}
Note that such a category has first been called maximal $(d-1)$-orthogonal by 
Iyama \cite{Iyama07a}.
\end{defn}

\begin{theo} \label{theorem.standard}
Let $\mathcal{T}$ be a triangulated category with an $(n-2)$-cluster tilting
subcategory $\mathcal{F}$, which is closed under $\Sigma_3^{n-2}$, where
$\Sigma_3$ denotes the suspension in $\mathcal{T}$.
Then $(\mathcal{F}, \Sigma_3^{n-2}, \foa)$ is an $n$-angulated category,
where $\foa$ is the class of all sequences
\[ X_1\xra{\alp_1} X_2\xra{\alp_2}\cdots\xra{\alp_{n-1}} X_n\xra{\alp_n} \Sigma_3^{n-2} X_1 \]
such that there exists a diagram
\[ \xymatrix@!=.5pc{ & X_2 \ar[rr]^{\alpha_2} \ar[rd] && X_3 \ar[rd] && \cdots && X_{n-1} \ar[rd]^{\alpha_{n-1}} \\ X_1 \ar[ru]^{\alpha_1} \ar@{<-{|-}}[rr] && X_{2.5} \ar[ru] \ar@{<-{|-}} [rr] && X_{3.5} & \cdots & X_{n-1.5} \ar[ru] \ar@{<-{|-}}[rr] && X_n} \]
with $X_i \in \mathcal{T}$ for $i \not \in \mathbb{Z}$, such that all oriented
triangles are triangles in $\mathcal{T}$, all non-oriented triangles 
commute, and $\alpha_n$ is the composition along the lower edge of the diagram.
\end{theo}

\begin{rems} \label{rems:std}
In the situation of Theorem~\ref{theorem.standard} we have the following:
\begin{itemize}
\item
The suspension in the $n$-angulated category $\mathcal{F}$
is given by $\Sigma_3^{n-2}$. In particular it should not be confused with
$\Sigma_3$.
\item
If $\Fo\subset\Tr$ is an $(n-2)$-cluster tilting subcategory, then
$\Fo$ is \emph{not} closed under $\Sig_3^h$ for $1\leq h<n-2$.
If $\mathcal{T}$ has a Serre functor $\mathbb{S}$, then the assumption
that $\mathcal{F}$ is closed under $\Sigma_3^{n-2}$ is equivalent to
$\mathcal{F}$ being closed under $\mathbb{S}$. (It follows immediately from
the definition of $(n-2)$-cluster tilting that $\mathcal{F}$ is closed under
$\mathbb{S} \Sigma_3^{-(n-2)}$.)
\end{itemize}
\end{rems}

The following observation will turn out to be helpful in the proof of
Theorem~\ref{theorem.standard}.

\begin{lem} \label{lemma.no_hom_to_shift}
In the notation of Theorem~\ref{theorem.standard}, we have
\[ {\mathcal{T}}(X_{i+1.5}, \Sigma_3^j \mathcal{F}) = 0 \text{ for }
0 < j < i < n-2\]
\end{lem}

\begin{proof}
We use downward induction on $i$. For $i = n-3$ consider the triangle
\[ X_{n-1.5} \ra X_{n-1} \ra X_n \ra \Sigma_3 X_{n-1.5}. \]
Since ${\mathcal{T}}(X_{n-1}, \Sigma_3^j \mathcal{F})
\subseteq {\mathcal{T}}(\mathcal{F}, \Sigma_3^j \mathcal{F}) = 0$
we have a monomorphism
\[ {\mathcal{T}}(X_{n-1.5}, \Sigma_3^j \mathcal{F}) \ra
{\mathcal{T}}(X_n, \Sigma_3^{j+1} \mathcal{F}) = 0, \]
so the claim of the lemma holds for $i=n-3$.

For arbitrary $i < n-3$ we use the triangle
\[ X_{i+1.5} \ra X_{i+2} \ra X_{i+2.5} \ra \Sigma_3 X_{i+1.5}. \]
As above we get a monomorphism
\[ {\mathcal{T}}(X_{i+1.5}, \Sigma_3^j \mathcal{F}) \ra
{\mathcal{T}}(X_{i+2.5}, \Sigma_3^{j+1} \mathcal{F}). \]
Now the latter space is $0$ inductively, so the claim holds.
\end{proof}

\begin{proof}[Proof of Theorem~\ref{theorem.standard}]
We verify the axioms.

(F1) (a) 
The class $\foa$ is closed under direct summands, as
one easily shows using the fact that, up to isomorphism,
the construction of a triangle over a morphism commutes
with the passage to direct summands. Part~(b) is clear.

For (F1)(c) let $\alpha_1$ be a morphism in $\mathcal{F}$. Let $X_{2.5}$ be
the cone of $\alpha_1$ in $\mathcal{T}$. Since $\mathcal{F} \subseteq
\mathcal{T}$ is $(n-2)$-cluster tilting $X_{2.5}$ has an
$\mathcal{F}$-coresolution of length at most $n-2$. Combining this coresolution
with the original triangle we obtain an $n$-angle starting with $\alpha_1$.

Next we verify axiom (F4). So assume we have the following commutative diagram.

\noindent
\begin{minipage}{\textwidth} \[ \xymatrix{
& X_2 \ar[rr]^{\alpha_2} \ar[rd]^{\alpha_2'} \ar[ddd]^(.6){\varphi_2} && X_3 \ar[rd]^{\alpha_3'} && \cdots && X_{n-1} \ar[rd]^{\alpha_{n-1}} \\
X_1 \ar[ru]^{\alpha_1} \ar[ddd]^(.4){\varphi_1} \ar@{<-{|-}}[rr]_(.3){\alpha_n^2} && X_{2.5} \ar[ru]^{\alpha_2''} \ar@{<-{|-}}[rr]_{\alpha_n^3} && X_{3.5} & \cdots & X_{n-1.5} \ar[ru]^{\alpha_{n-2}''} \ar@{<-{|-}}[rr]_{\alpha_n^{n-1}} && X_n \\
\\
& Y_2 \ar[rr]^{\beta_2} \ar[rd]^{\beta_2'} && Y_3 \ar[rd]^{\beta_3'} && \cdots && Y_{n-1} \ar[rd]^{\beta_{n-1}} \\
Y_1 \ar[ru]^{\beta_1} \ar@{<-{|-}}[rr]_{\beta_n^2} && Y_{2.5} \ar[ru]^{\beta_2''} \ar@{<-{|-}}[rr]_{\beta_n^3} && Y_{3.5} & \cdots & Y_{n-1.5} \ar[ru]^{\beta_{n-2}''} \ar@{<-{|-}}[rr]_{\beta_n^{n-1}} && Y_n \\
} \] \captionof{diagram}{Given morphisms} \end{minipage}

\noindent
The aim is to obtain a diagram

\noindent
\begin{minipage}{\textwidth} \[ \xymatrix@!C=.8cm{
& X_3 \oplus Y_2 \ar[rr]^{ \smp{ -\alpha_3 & \\ \varphi_3 & \beta_2}} \ar[rd] && X_4 \oplus Y_3 \ar[rd] && \cdots && X_n \oplus Y_{n-1} \ar[rd]^{\smp{ -\alpha_n & \\ \varphi_n & \beta_{n-1}}} \\
X_2 \oplus Y_1 \ar[ru]^{\smp{ -\alpha_2 & \\ \varphi_2 & \beta_1}} \ar@{<-{|-}}[rr] && H_{2.5} \ar[ru] \ar@{<-{|-}}[rr] && H_{3.5} & \cdots & H_{n-1.5} \ar[ru] \ar@{<-{|-}}[rr] && \Sigma_3^{n-2} X_1 \oplus Y_n \\
} \] \captionof{diagram}{Target diagram} \label{target_diagram} \end{minipage}

\noindent
such that the composition along the lower row is 
$\smp{ - \Sigma_3^{n-2} \alpha_1 & \\ \Sigma_3^{n-2} \varphi_1 & \beta_n}$.

We construct further morphisms $\varphi_i$, and at the same time the objects 
$H_i$ from left to right. (We actually also construct $\varphi_i$ for 
$i \in \{2.5, 3.5, \ldots, n-1.5\}$, even though they are not visible in the 
target diagram.)

We choose $\varphi_{2.5} \colon X_{2.5} \ra Y_{2.5}$ to be a ``good cone morphism'', that is a cone morphism such that the sequence
\[ X_2 \oplus Y_1 \xra{ \smp{ -\alpha_2' & \\ \varphi_2 & \beta_1 }} X_{2.5} \oplus Y_2 \xra{ \smp{ -\alpha_n^2 & \\ \varphi_{2.5} & \beta_2' }} \Sigma_3 X_1 \oplus Y_{2.5} \xra{ \smp{ - \Sigma_3 \alpha_1 & \\ \Sigma_3 \varphi_1 & \beta_n^2 }} \Sigma_3 X_2 \oplus \Sigma_3 Y_1 \]
is a triangle in $\mathcal{T}$.

Now the left-most triangle of the target diagram is the lower horizontal 
triangle in following octahedron.

\noindent
\begin{minipage}{\textwidth} \[ \xymatrix@C=2cm{
& \Sigma_3^{-1} X_{3.5} \ar[d]^(.4){\smp{ \Sigma_3^{-1} \alpha_n^3 \\ 0 }} \ar@{=}[r] & \Sigma_3^{-1} X_{3.5} \ar[d] & \\
X_2 \oplus Y_1 \ar@{=}[d] \ar[r]^{ \smp{ -\alpha_2' & \\ \varphi_2 & \beta_1 }} & X_{2.5} \oplus Y_2 \ar[d]^{\smp{ \alpha_2'' & \\ & \id }} \ar[r]^(.55){ \smp{ -\alpha_n^2 & \\ \varphi_{2.5} & \beta_2' }} & \Sigma_3 X_1 \oplus Y_{2.5} \ar@{-->}[d] \ar[r]^{ \smp{ -\Sigma_3 \alpha_1 & \\ \Sigma_3 \varphi_1 & \beta_n^2 }} & \Sigma_3 X_2 \oplus \Sigma_3 Y_1 \ar@{=}[d] \\
X_2 \oplus Y_1 \ar[r]^{ \smp{ -\alpha_2 & \\ \varphi_2 & \beta_1 }} & X_3 \oplus Y_2 \ar[d]^{(- \alpha_3' \; 0)} \ar@{-->}[r] & H_{2.5} \ar@{-->}[d] \ar@{-->}[r] & \Sigma_3 X_2 \oplus \Sigma_3 Y_1 \\
& X_{3.5} \ar@{=}[r] & X_{3.5}
} \] \captionof{diagram}{} \label{octahedron_1} \end{minipage}

We now iteratedly construct the further triangles of the target diagram 
(Diagram~\ref{target_diagram}). We assume to have a triangle
\[ \xymatrix@=1.5cm{ H_{i+1.5} \ar[r] & X_{i+2.5} \ar[rr]^-{\smp{- \Sigma_3^i \alpha_n^2 \cdots \Sigma_3 \alpha_n^{i+1} \alpha_n^{i+2} \\ \varphi_{i+1.5} \alpha_n^{i+2} }} && \Sigma_3^{i+1} X_1 \oplus \Sigma_3 Y_{i+1.5} \ar[r] & \Sigma_3 H_{i+1.5}.} \]
For $i = 1$ this can be found in the right column of the octahedron in 
Diagram~\ref{octahedron_1}, for $i > 1$ it is in the right column of the 
octahedron in Diagram~\ref{octahedron_3} for $i-1$. We use this triangle to 
construct the following octahedron

\noindent
\begin{minipage}{\textwidth} \[ \xymatrix@C=2cm{
& Y_{i+2} \ar@{-->}[d]_{\smp{0 \\ \id}} \ar@{=}[r] & Y_{i+2} \ar[d]^{\smp{0 \\ \beta_{i+2}'}} \\
H_{i+1.5} \ar@{=}[d] \ar@{-->}[r] & X_{i+2.5} \oplus Y_{i+2} \ar@{-->}[d]_{( \id \; 0)} \ar@{-->}[r]^{(*)} & \Sigma_3^{i+1} X_1 \oplus Y_{i+2.5} \ar[d]^{\smp{\id & \\ & \beta_n^{i+2}}} \ar[r] & \Sigma_3 H_{i+1.5} \ar@{=}[d] \\
H_{i+1.5} \ar[r] & X_{i+2.5} \ar[d]_{0} \ar[r]^-{\smp{- \Sigma_3^i \alpha_n^2 \cdots \Sigma_3 \alpha_n^{i+1} \alpha_n^{i+2} \\ \varphi_{i+1.5} \alpha_n^{i+2} }} & \Sigma_3^{i+1} X_1 \oplus \Sigma_3 Y_{i+1.5} \ar[d]^{(0 \; - \Sigma_3 \beta_{i+1}'')} \ar[r] & \Sigma_3 H_{i+1.5} \\
& \Sigma_3 Y_{i+2} \ar@{=}[r] & \Sigma_3 Y_{i+2}
} \] \captionof{diagram}{} \label{octahedron_2} \end{minipage}

Note that the map $X_{i+2.5} \ra \Sigma_3 Y_{i+2}$ vanishes by 
Lemma~\ref{lemma.no_hom_to_shift}. Hence the left vertical triangle splits, 
and thus can be assumed to have the form above.

We dermine the shape of the map $(*)$: By commutativity of the upper square we 
know that the right column of the matrix $(*)$ is $\smp{0 \\ \beta_{i+2}'}$, 
and by commutativity of the central square we know that the left upper entry of
 the matrix $(*)$ is $- \Sigma_3^i \alpha_n^2 \cdots \alpha_n^{i+2}$. Hence
\[ (*) = \begin{pmatrix} - \Sigma_3^i \alpha_n^2 \cdots \alpha_n^{i+2} & 0 \\ \varphi_{i+2.5} & \beta_{i+2}' \end{pmatrix} \]
for some $\varphi_{i+2.5} \colon X_{i+2.5} \ra Y_{i+2.5}$.

Now we take the upper horizontal triangle of the above octahedron to construct 
the following.

\noindent
\begin{minipage}{\textwidth} \[ \xymatrix@C=1.3cm{
& \Sigma_3^{-1} X_{i+3.5} \ar[d]_{\smp{ \Sigma_3^{-1} \alpha_n^{i+3} \\ 0 }} \ar@{=}[rr] && \Sigma_3^{-1} X_{i+3.5} \ar[d] \\
H_{i+1.5} \ar@{=}[d] \ar[r] & X_{i+2.5} \oplus Y_{i+2} \ar[d]_{\smp{\alpha_{i+2}'' & \\ & \id}} \ar[rr]^{\smp{ - \Sigma_3^i \alpha_n^2 \cdots \alpha_n^{i+2} & 0 \\ \varphi_{i+2.5} & \beta_{i+2}'} } && \Sigma_3^{i+1} X_1 \oplus Y_{i+2.5} \ar@{-->}[d] \ar[r] & \Sigma_3 H_{i+1.5} \ar@{=}[d] \\
H_{i+1.5} \ar[r] & X_{i+3} \oplus Y_{i+2} \ar[d]_{(- \alpha_{i+3}' \; 0)} \ar@{-->}[rr] && H_{i+2.5} \ar@{-->}[d] \ar@{-->}[r] & \Sigma_3 H_{i+1.5} \\
& X_{i+3.5} \ar@{=}[rr] && X_{i+3.5}
}\] \captionof{diagram}{} \label{octahedron_3} \end{minipage}

The lower row of this diagram is supposed to be the $(i+1)$-st triangle of 
the target diagram. We need to check that the composition
\begin{align*}
\tag{$\dagger$} & X_{i+2} \oplus Y_{i+1} \ra H_{i+1.5} \ra X_{i+3} \oplus Y_{i+2}
\intertext{is of the form $\smp{-\alpha_{i+2} & \\ \varphi_{i+2} & \beta_{i+1}}$ for some $\varphi_{i+2} \colon X_{i+2} \ra Y_{i+2}$. First note that by commutativity of the left square of Diagram~\ref{octahedron_3} the above composition is}
= & \smp{\alpha_{i+2}'' & \\ & \id} \circ \Big[ X_{i+2} \oplus Y_{i+1} \ra H_{i+1.5} \ra X_{i+2.5} \oplus Y_{i+2} \Big]
\end{align*}
We first focus on the components of $(\dagger)$ starting in $Y_{i+1}$. By 
commutativity of the central squares of Diagrams~\ref{octahedron_1} (for $i=1$)
and \ref{octahedron_3} (for $i > 1$) we have that these components can be 
rewritten
\[ \smp{\alpha_{i+2}'' & \\ & \id} \circ \Big[ \Sigma_3^i X_1 \oplus Y_{i+1.5} \ra H_{i+1.5} \ra X_{i+2.5} \oplus Y_{i+2} \Big] \circ \smp{0 \\ \beta_{i+1}'}, \]
and further, by anti-commutativity of the outer square of 
Diagram~\ref{octahedron_2}, as
\[ \smp{\alpha_{i+2}'' & \\ & \id} \circ \smp{ 0 \\ \id } \circ (0 \; \beta_{i+1}'') \circ  \smp{0 \\ \beta_{i+1}'} = \smp{0 \\ \beta_{i+1}}. \]
Next we look at the components of $(\dagger)$ ending in $X_{i+3}$. That is
\begin{align*}
& (\alpha_{i+2}'' \; 0) \circ \Big[ X_{i+2} \oplus Y_{i+1} \ra H_{i+1.5} \ra X_{i+2.5} \oplus Y_{i+2} \Big] \\
= & \alpha_{i+2}'' \circ \Big[ X_{i+2} \oplus Y_{i+1} \ra H_{i+1.5} \ra X_{i+2.5} \Big].
\end{align*}
By commutativity of the lower squares of Diagrams~\ref{octahedron_1} 
(for $i=0$) and \ref{octahedron_3} (for $i>1$) this is
\[ \alpha_{i+2}'' \circ (- \alpha_{i+2}' \; 0) = ( - \alpha_{i+2} \; 0 ). \]
Thus we have now shown that the composition $(\dagger)$ is of the form 
$\smp{-\alpha_{i+2} & \\ \varphi_{i+2} & \beta_{i+1}}$ as claimed.

This iteration works for $i = 1, \ldots, n-4$ (for $i = n-4$ we set 
$X_{n-0.5} := X_n$ and $\alpha_{n-1}' := \alpha_{n-1}$ in 
Diagram~\ref{octahedron_3}). With these (and similar) identifications 
Diagram~\ref{octahedron_2} can still be constructed for $i = n-3$. 
Its upper horizontal triangle is
\[ \xymatrix@R=.5cm{
H_{n-1.5} \ar[r] & X_{n-0.5} \oplus Y_{n-1} \ar@{=}[d] \ar[rrr]^{\smp{ - \Sigma_3^{n-3} \alpha_n^2 \cdots \alpha_n^{n-1} & 0 \\ \varphi_{n-0.5} & \beta_{n-1}' }} &&& \Sigma_3^{n-2} X_1 \oplus Y_{n-0.5} \ar@{=}[d] \ar[r] & \Sigma_3 H_{n-1.5} \\
& X_n \oplus Y_{n-1} \ar[rrr]_{\smp{ - \alpha_n & 0 \\ \varphi_n & \beta_{n-1}}} &&& \Sigma_3^{n-2} X_1 \oplus Y_n
} \]
This is the final triangle of the target diagram (Diagram~\ref{target_diagram}).

It remains to check that the composition
\[ \Sigma_3^{n-2} X_1 \oplus Y_n \ra \Sigma_3 H_{n-1.5} \ra \Sigma_3^2 H_{n-2.5} \ra \cdots \ra \Sigma_3^{n-3} H_{2.5} \ra \Sigma_3^{n-2} X_2 \oplus \Sigma_3^{n-2} Y_1 \]
is $\smp{ - \Sigma_3^{n-2} \alpha_1 & 0 \\ \Sigma_3^{n-2} \varphi_1 & \beta_n}$. By commutativity of the right squares of Diagrams~\ref{octahedron_2} and \ref{octahedron_1} (for various $i$) we iteratedly obtain
\begin{align*}
= &  \Big[ \Sigma_3^{n-2} X_1 \oplus \Sigma_3 Y_{n-1.5} \ra \Sigma_3 H_{n-1.5} \ra \cdots \ra \Sigma_3^{n-3} H_{2.5} \ra \Sigma_3^{n-2} X_2 \oplus \Sigma_3^{n-2} Y_1 \Big] \circ \smp{ \id & \\ & \beta_n^{n-1} } \\
= &  \Big[ \Sigma_3^{n-2} X_1 \oplus \Sigma_3 Y_{n-1.5} \ra \Sigma_3^2 H_{n-2.5} \cdots \ra \Sigma_3^{n-3} H_{2.5} \ra \Sigma_3^{n-2} X_2 \oplus \Sigma_3^{n-2} Y_1 \Big] \circ \smp{ \id & \\ & \beta_n^{n-1} } \\
= &  \Big[ \Sigma_3^{n-2} X_1 \oplus \Sigma_3^2 Y_{n-2.5} \ra \Sigma_3^2 H_{n-2.5} \cdots \ra \Sigma_3^{n-3} H_{2.5} \ra \Sigma_3^{n-2} X_2 \oplus \Sigma_3^{n-2} Y_1 \Big] \circ \smp{ \id & \\ & \Sigma_3 \beta_n^{n-2} \beta_n^{n-1} } \\
& \qquad \vdots \\
= & \Big[ \Sigma_3^{n-2} X_1 \oplus \Sigma_3^{n-3} Y_{2.5} \ra \Sigma_3^{n-3} H_{2.5} \ra \Sigma_3^{n-2} X_2 \oplus \Sigma_3^{n-2} Y_1 \Big] \circ \smp{ \id & \\ & \Sigma_3^{n-4} \beta_n^3 \cdots \beta_n^{n-1} }
\intertext{and by commutativity of the right square of Diagram~\ref{octahedron_1} this is}
= & \smp{ - \Sigma_3^{n-2} \alpha_1 & \\ \Sigma_3^{n-2} \varphi_1 & \Sigma_3^{n-3} \beta_n^2} \circ \smp{ \id & \\ & \Sigma_3^{n-4} \beta_n^3 \cdots \beta_n^{n-1} } \\
= & \smp{ - \Sigma_3^{n-2} \alpha_1 & \\ \Sigma_3^{n-2} \varphi_1 & \beta_n}
\end{align*}
This completes the verification of axiom (F4).

Axiom (F3) follows from (F4) immediately.

To see that axiom (F2) holds first note that the $\foa$ is clearly closed under
shifts by $\pm 2n$. Hence it suffices to show the ``only if''-part of the 
claim. This follows from (F4) for $Y_i = 0$.
\end{proof}

\section{Calabi-Yau Properties}
\label{s:calabi-yau-properties}
\subsection{Serre functor and Calabi-Yau categories}
Let $\ka$ be a field and $\Fo$  a $\ka$-category. For a self-equivalence $\Se$
of $\Fo$ and $M\in\rmd{\Fo}$ we denote by $\Se M$ the twisted module
$M\circ \Se^{-1}$.
Let us consider the $\Fo$-bimodules $D\Fo$ and $\Se \Fo$ which are defined
as follows:
\[
D\Fo(x,y):= \Hom_\ka(\Fo(y,x),\ka) \text{ and }
\Se\Fo(x,y):=\Fo(\Se^{-1}x,y) \text{ for all } x,y\in\Fo.
\]
We say that $\Se$ is a {\em Serre-functor} if $\Se\Fo\cong D\Fo$
as $\Fo$-bimodules.

Suppose that $\Tr$ is a triangulated category with suspension $\Sig$ such that
the $\ka$-linear category $\Tr$ admits a Serre functor $\Se$ as above. Then, as
shown in~\cite{BonKap89} and~\cite[Appendix]{Bocklandt08}, $\Se$ becomes a
canonically a triangle functor
$\Tr\ra\Tr$. We say that $\Tr$ is $n$-{\em Calabi-Yau} if
$\Se$ is isomorphic to $\Sig^n$ as a triangle functor.

\subsection{Serre functor for stable module categories}
Let $\Fo$ be a $\ka$-category with a Serre functor $\Se$ as above, and suppose
that the category $\rmd{\Fo}$ of finitely presented functors $\Fo\ra\rmd{\ka}$
is abelian. Then it is automatically a
Frobenius category and $-\otimes_\Fo D\Fo \cong \Se$ is a self-equivalence
of $\rmd{\Fo}$. In fact, $\Se$ is a Nakayama
functor \ie we have a functorial homomorphism
\[
\psi_Y\df D\Hom_\Fo(Y,?)\ra \Hom_\Fo(?,\Se Y)
\]
which is an isomorphism for $Y\in\rmd{\Fo}$ projective.
It follows then from the lemma below that the triangle functor
\[
\Ome_{\smd{\Fo}}\circ \Se
\]
is a Serre functor for the stable category $\smd{\Fo}$ which is naturally
a triangulated category~\cite{Heller68a, happ87a}
with suspension functor $\Ome_{\smd{\Fo}}^{-1}$.

\begin{lem}\label{lem:SF}
Let  $\cE$ be a Hom-finite Frobenius category and denote by $\Ome^{-1}$
the suspension functor for the stable category $\sE$.
Let $\cP\subset \cE$ be the full subcategory of
projective-injectives. Suppose that $\nu\df\cP\xra{\sim}\cP$
is a Serre functor for $\cP$ such that $D\cE(P,X)\cong\cE(X, \nu P)$
holds functorially for $P\in\cP$ and $X\in\cE$.
Then $\nu$ induces a self-equivalence of
$\sE$ which we denote also by $\nu$ and the composition
$S_\sE:=\Ome\circ\nu$ is a Serre functor for the stable category $\sE$ as a
triangle functor.
\end{lem}

\begin{proof}
Let $\Cac(\cP)$ be the dg-category of acyclic complexes of projectives from
$\cP$, and $\Hac(\cP)=H^0(\Cac(\cP))$ its homotopy category. It follows
from~\cite{kevo88} that we have an equivalence of triangulated categories
\[
Z^0\df\Hac(\cP)\ra \sE
\]
which sends a complex $P^\bullet=(\cdots\ra P^{-1}\ra P^0\ra P^1\ra \cdots)$
to $Z^0P^\bullet$. Since $\cE$ has enough projectives,
a complex $P^\bullet$ in $\cE$ is acyclic if and only if $\cE(Q,P^\bullet)$
is acyclic for any $Q\in\cP$. Dually, it is acyclic if and only if
$\cE(P^\bullet,I)$ is acyclic for any $I\in\cP$. It follows, that
$\nu$ defines a selfequivalence of $\Hac(\cP)$. We denote also by $\nu$
the induced selfequivalence of $\sE$. Now, using the isomorphism
$D\cE(P,X)\cong\cE(X, \nu P)$, it is easy to verify
for $P^\bullet,Q^\bullet\in\Hac(\cP)$ the following isomorphisms
\begin{multline*}
D\Hac(\cP)(P^\bullet,Q^\bullet) \cong DH^0\cE(P^\bullet[-1],Z^0(Q^\bullet))\\
\cong H^0\cE(Z^0(Q^\bullet),\nu P^\bullet[-1])
\cong \Hac(\cP)(Q^\bullet,\nu P^\bullet[-1]),
\end{multline*}
which imply our claim.
\end{proof}

\subsection{Construction of Calabi-Yau categories from $n$-angulated
categories} \label{subsect.CY}
Consider now an $n$-angulated $\ka$-category  $\Fo=(\Fo,\Sig_n,\foa)$
which has a Serre-functor $\Se$ with $\Se=\Sig_n^d$ as $\ka$-linear functors.
Then the triangulated category
$\smd{\Fo}$ is $(nd-1)$-Calabi-Yau. In fact, by  Lemma~\ref{lem:class1}
in $\smd{\Fo}$, we have an isomorphism of triangle functors
$\Ome^{-n}_{\smd{\Fo}} \cong \Sig_n$.
Finally, by Lemma~\ref{lem:SF}, $\smd{\Fo}$ has as a Serre functor as a
triangulated category:
\[
\Ome_{\smd{\Fo}}\circ\Se\cong \Ome_{\smd{\Fo}}\Sig_n^d\cong\Ome_{\smd{\Fo}}^{1-nd}
\]
with all isomorphisms being isomorphisms of triangle functors.

\section{Examples}
\label{s:examples}

\subsection{$n$-angulated subcategories of derived categories}

Let $\Lambda$ be a finite dimensional algebra of global dimension at 
most $n-2$. We denote the Serre functor of ${\rm D^b}(\rmd \Lambda)$ by $\Se$.

\begin{defn}[see \cite{IyamaOppermann09a}]
The algebra $\Lambda$ is called $(n-2)$-representation finite if the module 
category $\rmd \Lambda$ has an $(n-2)$-cluster tilting object. (Cluster 
tilting subcategories of abelian categories are defined in the same way as 
cluster tilting subcategories for triangulated categories in 
Definition~\ref{def.ct}.)
\end{defn}

In this case, by \cite[Theorem~1.23]{Iyama08}, the subcategory
\[ \mathcal{U} = \add \{ \Se^i \Lambda [-(n-2)i] \mid i \in \mathbb{Z} \} \]
of the bounded derived category is $(n-2)$-cluster tilting. Moreover, 
by~\cite[Corollary~3.7]{IyamaOppermann09}, $\Se \mathcal{U} = \mathcal{U}$.

Thus, by Theorem~\ref{theorem.standard}, $\mathcal{U}$ is an $n$-angulated 
category for any $(n-2)$-representation finite algebra $\Lambda$.

\subsection{$n$-angulated categories in Calabi-Yau categories}
\label{s:n-ang-cat-in-CY-cat}

Let $\Tr$ be a triangulated $d$-Calabi-Yau-category with Serre
functor $\Se$. By our standard construction, an $(n-2)$-cluster tilting
subcategory $\Fo\subset\Tr$ which
is closed under $\Sig^{n-2}$ has a structure
of a $n$-angulated category with suspension $\Sig_n:=\Sig^{n-2}$.
By the remarks in~\ref{rems:std} we conclude that $d=d'(n-2)$ for some
integer $d'\in\Zet$, and in particular $\Sig_n^{d'}=\Se$.
By Remark~\ref{subsect.CY} we conclude that $\smd{\Fo}$ is 
$(d+2d'-1)$-Calabi-Yau.

By arguments in \cite{0705.3903} for $d=2$ we have that the endomorphism ring 
of a $2$-cluster tilting object $T$ is selfinjective if and only if 
$\Se T \in \add(T)$. Thus we see the following.

\begin{prp}
Let $\Lambda$ be a selfinjective $2$-Calabi-Yau tilted algebra. Then 
$\rpr \Lambda$ is a $4$-angulated category, with the Nakayama functor $\nu$ 
as suspension.
\end{prp}

There are many examples for this proposition which are closely related to 
cluster algebras:
\begin{itemize}
\item
By definition, a selfinjective cluster tilted algebra $A$ is a selfinjective
algebra of the form $A=\End_{\cC_H}(T)$ for $T$ a cluster tilting object
$T$ in a cluster category $\cC_H$ associated to an hereditary algebra $H$.
It is well-known, that $A$ is selfinjective if and only if $T=\Sig^2 T$.
Thus, by our standard construction, the category $\add(T)\cong\rpr(A)$ is
a 4-angulated category with suspension $\Sig_4=\Sig^2$ being the Serre functor.
Note that the selfinjective cluster tilted algebras were classified
by Ringel~\cite{0705.3903}.

\item
Quite similarly, one may study selfinjective cluster quasi-tilted
algebras of the form $A=\End_{\cC_\Xx}(T)$ for a cluster tilting
object $T$ in the cluster category $\cC_\Xx$ of a weighted
projective line $\Xx$ in the sense of Geigle and
Lenzing~\cite{GeLe87}. A similar analysis as in~\cite{0705.3903}
shows, that this can occur only if $\Xx$ is of one of the following
tubular types: $(6,3,2)$, $(4,4,2)$ and $(2,2,2,2)$. Note that since
in $\cC_\Xx$ we have $\tau=\Sig$, for the tubular type $(3,3,3)$ no
cluster tilting object $T$ can fulfill $\Sig^2 T=T$ because
$\tau^3=\id$. In contrast, by the same token in type $(2,2,2,2)$ all
cluster tilting objects fulfill $\Sig^2T=T$. In this case, one
obtains 4 families of weakly symmetric  cluster quasi-tilted
algebras which are best described using quivers with potential, cf.\
figure~\ref{fig:fig1}.

\begin{figure}
\begin{tabular}{ll}
$Q_1=\vcenter{\xymatrix@+1.3pc{
x_0\ar@{.}[d] & \ar[l]|(.4){a_{00}}\ar[ld]|(.35){a_{10}} y_0
              & \ar[l]|(.4){b_{00}}\ar[ld]|(.35){b_{10}} z_0
              & \ar[l]|(.4){c_{00}}\ar[ld]|(.35){c_{10}} x_0\ar@{.}[d]\\
x_1           & \ar[l]|(.4){a_{11}}\ar[lu]|(.35){a_{01}} y_1
              & \ar[l]|(.4){b_{11}}\ar[lu]|(.35){b_{01}} z_1
              & \ar[l]|(.4){c_{11}}\ar[lu]|(.35){c_{01}} x_1\\
}}$
&
\begin{minipage}[c]{5.8cm}
\begin{multline*}
W_1  = \lam a_{00}b_{00}c_{00} - a_{01}b_{10}c_{00}\\
           +a_{11}b_{10}c_{01} - a_{10}b_{00}c_{01}\\
           +a_{01}b_{11}c_{10} - a_{00}b_{01}c_{10}\\
           +a_{10}b_{01}c_{11} - a_{11}b_{11}c_{11},
\end{multline*}
\end{minipage}\\
$Q_2=\vcenter{\xymatrix@C+0.2pc{
w_0\ar@{.}[dd] & & \ar[ld]|{b_0}y_0 & & \ar[ld]|{d_0} w_0\ar@{.}[dd]\\
    &\ar[lu]|{a_0}\ar[ld]|{a_1} x & & \ar[lu]|{c_0}\ar[ld]|{c_1} z&\\
w_1 & & \ar[lu]|{b_1}y_1 & & \ar[lu]|{d_1} w_1
}}$
&
\begin{minipage}[c]{5.8cm}
\begin{multline*}
W_2 =\lam a_0b_0c_0d_0 - a_0b_1c_1d_0\\
   + a_1b_1c_1d_1 - a_1b_0c_0d_1,
\end{multline*}
\end{minipage}\\
$Q_3=\vcenter{\xymatrix@C+0.3pc{
\ar@{.}[dd]&\ar[ld]|{a} x & & \ar[ld]|{c_0}z_0 &\ar@{.}[dd]  \\
w&&\ar[lu]|{b}\ar[ll]|{\tilde{a}} y &\ar[l]|{c_1}z_1 & \ar[lu]|{d_0}\ar[l]|{d_1} \ar[ld]|{d_2}w\\
 & &&\ar[lu]|{c_2}z_2 &
}}$
&
\begin{minipage}[c]{5.8cm}
\begin{multline*}
W_3= \lam \tilde{a}c_0d_0 - abc_0d_0\\
    + abc_1d_1-\tilde{a}c_1d_1+\tilde{a}c_2d_2,
\end{multline*}
\end{minipage}
\end{tabular}

\begin{tabular}{ll}
$Q_4=\vcenter{\xymatrix@C+1.5pc@R-0.2pc{
 &\ar[ldd]|{a_0} y_0 & & \\
\ar@{.}[dd] &\ar[ld]|{a_1} y_1 & &\ar@{.}[dd]\\
x&\ar[l]|{a_2} y_2&\ar[luu]|{b_0}\ar[lu]|{b_1}\ar[l]|{b_2}\ar[ld]|{b_3} z& \ar@/_1pc/[l]|{c_0}\ar@/^1pc/[l]|{c_1} x\\
&\ar[lu]|{a_3} y_3 & &
}}$
&
\begin{minipage}[c]{5.8cm}
\begin{multline*}
W_4 = \lam a_0b_0c_0 + a_1b_1c_0 - a_2b_2c_0\\
      - a_0b_0c_1 - a_1b_1c_1 + a_3b_3c_1.
\end{multline*}
\end{minipage}
\end{tabular}
\caption{Quivers with potentials for $4$ families of weakly
symmetric  cluster quasi-tilted algebras} \label{fig:fig1}
\end{figure}
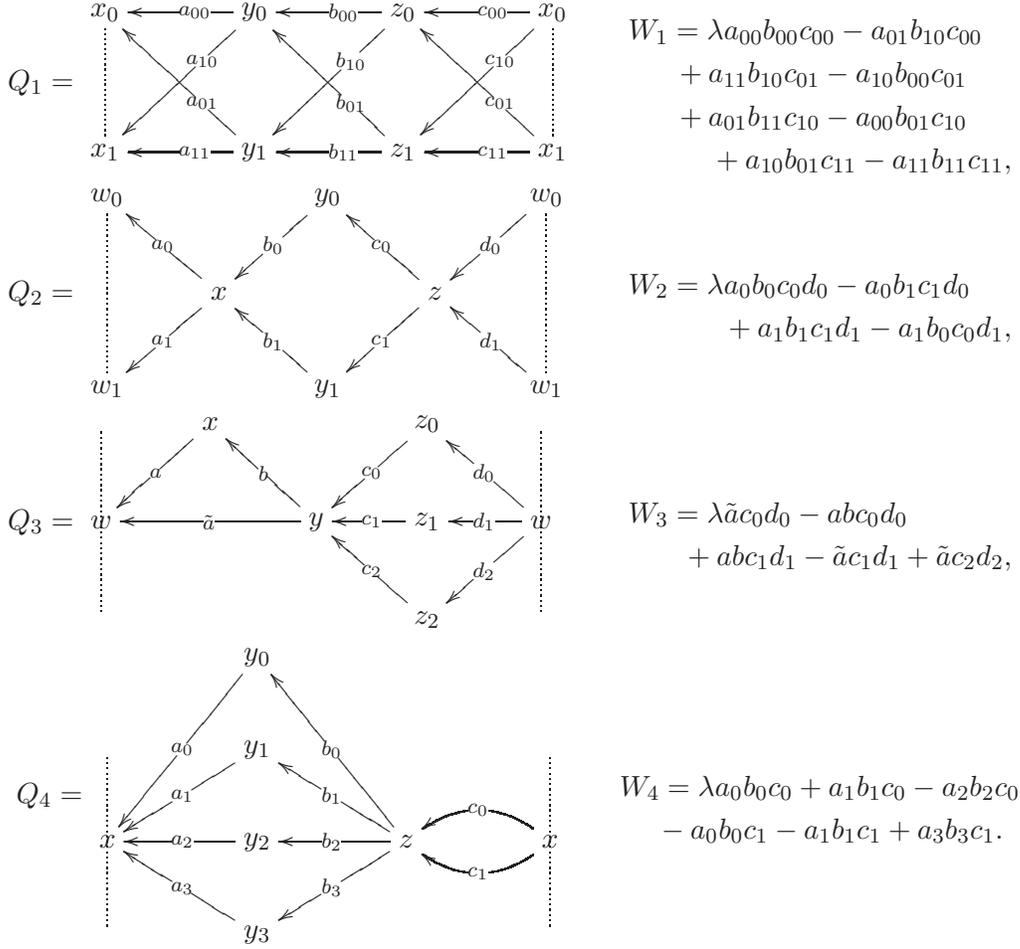
Here, the parameter $\lam$ runs through $\ka\setminus\{0,1\}$ and
the left and right borders (marked by a dotted line) of each
quiver have to be identified. Recall that for weight type $(2,2,2,2)$
the cluster category $\cC_\Xx$ depends on a parameter. In the corresponding
$4$-angulated categories with four isomorphism classes of indecomposable 
objects,
the suspension $\Sig_4$ as well as the Serre functor are
the identity (which is a bit untypical).

\item
Let $\Lam$ be the preprojective algebra of Dynkin type $\dA,\dD,\dE$.
It is well-known that the category of projective modules $\rpr{\Lam}$
is a 1-Calabi-Yau triangulated category. Thus, the stable module category
$\smd{\Lam}$ is a 2-Calabi-Yau triangulated category. In fact, it is
a generalized cluster category in the sense of Amiot associated to a
stable Auslander algebra of the corresponding Dynkin type~\cite{Amiot09}.
With the description
in~\cite{GeLeSc05b} it is easy to check that for type $\dA_n$,  the standard
cluster tilting object $T_n$ associated to the linear orientation of $\dA_n$
fulfills $\Ome_\Lam^{-2} T_n\cong T_n$.
Thus, $A_n:=\End_\Lam(T_n)$ is selfinjective,
and $\rpr{A_n}\cong\add(T_n)$ is a 4-angulated category. Note, that for
$n\geq 6$ the algebras $A_n$ are not cluster (quasi-)tilted.
\end{itemize}

For $d > 2$ it is in general insufficient to check selfinjectivity
of the cluster tilted algebra in order to find out if the cluster
tilting object is closed under the Serre functor.

If $\Lambda$ is $(n-2)$-representation finite then we still know (by
\cite[Corollary~3.7]{IyamaOppermann09}) that the canonical $(n-2)$-cluster
tilting object $\pi \Lambda$ in the corresponding $(n-2)$-Amiot
cluster category $\mathcal{C}$ is closed under the Serre functor.
Thus we have that $\add(\pi \Lambda) = \rpr \End_{\mathcal{C}}(\pi
\Lambda)$ is an $n$-angulated category.

These $n$-angulated categories (in particular for the case of higher
Auslander algebras of linearly oriented $A_n$) are studied in
\cite[Sections~5 and 6]{OppTho10}.

\subsection{Link to algebraic geometry and string theory} 
\label{ss:links-alg-geom-string-theory}
The theory of non-commutative Donaldson-Thomas invariants developed by 
Kontsevich-Soibelman
in \cite{KontsevichSoibelman08, KontsevichSoibelman10} allows one to associate
refined DT-invariants to suitable $3$-Calabi-Yau categories.
In particular, for each finite quiver $Q$ endowed with a
potential $W$ which is a {\em finite} linear combination of cycles
of length at least three, this theory provides us with a refined 
DT-invariant $A_{Q,W}$, which is an element of a quantized power series algebra
$\Lambda_Q$ depending only on $Q$.
In string theory, the element $A_{Q,W}$ is also 
called~\cite{CecottiNeitzkeVafa10}
the {\em Kontsevich-Soibelman half-monodromy}
and its square the {\em Kontsevich-Soibelman monodromy}. The
determination of its spectrum (in representations of $\Lambda_Q$) is of
considerable interest. An easy case is the one where $A_{Q,W}$ is
of finite order. Many examples where this occurs are related to
$4$-angulated categories. Indeed, let us suppose that $(Q,W)$ satisfies suitable
finiteness conditions and in particular that its Jacobi algebra
is finite-dimensional.  Then one can show (see \cite{Keller10b}) that 
the Kontsevich-Soibelman half-monodromy (more precisely: its adjoint action
composed with the antipodal map) is `categorified' by the
suspension functor $\Sigma$ of the generalized cluster-category $\cc_{Q,W}$
associated with $(Q,W)$ in \cite{Amiot09}. Thus, the Kontsevich-Soibelman 
monodromy is
categorified by $\Sigma^2$. If the cluster category $\Tr=\cc_{Q,W}$ contains
a $2$-cluster tilting subcategory $\Fo$ invariant under $\Sigma^2$, then
the action of $\Sigma^2$ on the set of isoclasses of indecomposables
in $\Fo$ defines a faithful permutation representation of the cyclic group 
generated by the Kontsevich-Soibelman monodromy. In particular, in this case, 
the order of the Kontsevich-Soibelman monodromy is finite. The case where
$\Fo$ is stable under $\Sigma^2$ is precisely the one where $\Fo$ inherits
a natural structure of $4$-angulated category with $4$-suspension functor
given by $\Sigma^2$. Thus, the action of the $4$-suspension functor in
the $4$-angulated category $\Fo$ categorifies the Kontsevich-Soibelman
monodromy.

As a particularly nice example (already mentioned in
Section~\ref{s:n-ang-cat-in-CY-cat}), let us consider the quiver $Q$
below endowed with the potential $W$ which is the signed sum over all
small triangles, where adjacent triangles have different signs.
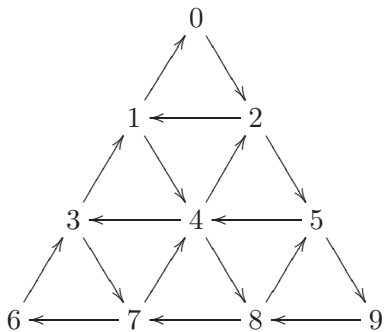
\[
\begin{xy} 0;<0.5pt,0pt>:<0pt,-0.5pt>::
(138,0) *+{0} ="0",
(91,76) *+{1} ="1",
(183,76) *+{2} ="2",
(45,153) *+{3} ="3",
(138,153) *+{4} ="4",
(229,153) *+{5} ="5",
(0,229) *+{6} ="6",
(91,229) *+{7} ="7",
(183,229) *+{8} ="8",
(274,229) *+{9} ="9",
"1", {\ar"0"},
"0", {\ar"2"},
"2", {\ar"1"},
"3", {\ar"1"},
"1", {\ar"4"},
"4", {\ar"2"},
"2", {\ar"5"},
"4", {\ar"3"},
"6", {\ar"3"},
"3", {\ar"7"},
"5", {\ar"4"},
"7", {\ar"4"},
"4", {\ar"8"},
"8", {\ar"5"},
"5", {\ar"9"},
"7", {\ar"6"},
"8", {\ar"7"},
"9", {\ar"8"},
\end{xy}
\]
As shown in \cite{GeLeSc05b},
the category of projective modules over the associated Jacobi algebra
is equivalent to a $2$-cluster tilting subcategory $\Fo$ of the stable
module category over the preprojective algebra of type $\dA_5$.
By \cite{Amiot09}, this category can also be described as the generalized
cluster category associated with the above pair $(Q,W)$. The suspension
functor of this category is of order $6$ and its square leaves
$\Fo$ invariant. The associated permutation of the indecomposable
projectives corresponds to the rotation by $120$ degrees of the
above quiver. We conclude that the $4$-angulated category $\Fo$
has a $4$-suspension functor of order~$3$ and that this is also
the order of the Kontsevich-Soibelman monodromy
associated with $(Q,W)$.

We refer to \cite{Keller10b} for
a more detailed survey of the ideas sketched above and for
many more examples.


\newcommand{\SortNoop}[1]{}\def\cprime{$'$} \def\cprime{$'$}
\providecommand{\bysame}{\leavevmode\hbox to3em{\hrulefill}\thinspace}

\end{document}